\documentclass[11pt]{amsart}
\usepackage{amssymb}
\usepackage{dsfont}

\newtheorem{theorem}{Theorem}[section]
\newtheorem{lemma}[theorem]{Lemma}
\newtheorem{proposition}[theorem]{Proposition}

\newtheorem{claim}[theorem]{Claim}

\newcommand{\N}{\mathbb N}
\newcommand{\Q}{\mathbb Q}

\newcommand{\R}{\mathbb R}

\newcommand{\C}{\mathbb C}

\newcommand{\on}{\operatorname}
\author{Artur Bartoszewicz}

\address{Institute of Mathematics, \L\'od\'z University of Technology, W\'olcza\'nska 215, 93-005
\L\'od\'z, Poland}
\email {arturbar@p.lodz.pl}

\author{Szymon G\l \c ab}

\address{Institute of Mathematics, \L\'od\'z University of Technology, W\'olcza\'nska 215, 93-005
\L\'od\'z, Poland}
\email {szymon.glab@p.lodz.pl}

\author{Adam Paszkiewicz}

\address{Faculty of Mathematics and Computer Science, University of \L\'od\'z, Banacha 22,
90-238 \L\'od\'z, Poland}
\email {ktpis@math.uni.lodz.pl}

\thanks{The first two authors have been supported by the Polish Ministry of Science and Higher Education Grant No.  N N201 414939 (2010--2013).}
\title[Large free linear algebras of real and complex functions]{Large free linear algebras of real and complex functions}
\date{}

\begin{document}

\begin{abstract}
Let $X$ be a set of cardinality $\kappa$ such that $\kappa^\omega=\kappa$. We prove that the linear algebra $\R^X$ (or $\C^X$) contains a free linear algebra with $2^\kappa$ generators. Using this, we prove several algebrability results for spaces $\C^\C$ and $\R^\R$. In particular, we show that the set of all perfectly everywhere surjective functions $f:\C\to\C$ is strongly $2^\mathfrak{c}$-algebrable. We also show that the set of all functions $f:\R\to\R$ whose sets of continuity points equals some fixed $G_\delta$ set $G$ is strongly $2^\mathfrak{c}$-algebrable if and only if $\R\setminus G$ is $\mathfrak{c}$-dense in itself. 
\end{abstract}

\maketitle

\section{Introduction}

Let $A$ be a linear algebra and let $E\subset A$. It is natural to ask whether $E\cup\{0\}$ contains an infinitely generated linear subalgebra $A'$ of $A$. If it is so, we say that $E$ is algebrable. There are numerous papers which are devoted to the algebrability of sets $E$ which are far from being linear, that is $x,y\in E$ does not in general imply $x+y\in E$. 

For a cardinal $\kappa$, a subset $E$ of a linear algebra $A$ is called $\kappa$-algebrable whenever $E \cup \{0\}$ contains a $\kappa$-generated linear algebra. If $E$ is $\omega$-algebrable, it is simply said to be algebrable. Let us observe that the set $E$ is $\kappa$-algebrable for $\kappa>\omega$ if and only if it contains an algebra which is a $\kappa$-dimensional linear space (see \cite{BG}). Additionally, we say that a subset $E$ of a commutative linear algebra $A$ is strongly $\kappa$-algebrable \cite{BG} if there exists a $\kappa$-generated free algebra $A'$ contained in $E \cup \{0\}$. Note, that $X= \{x_\alpha : \alpha < \kappa \} \subset E$ is a set of free generators of a free algebra $A' \subset E\cup\{0\}$ if and only if the set $X'$ of elements of the form $x_{\alpha_1}^{k_1} x_{\alpha_2}^{k_2} ... x_{\alpha_n}^{k_n}$ is linearly independent and all linear combinations of elements from $X'$ are in $E \cup \{0 \}$. It is easy to see that free algebras have no divisors of zero.

In practice, to prove the $\kappa$-algebrability of a set $E \subset A$, we have to find a linearly independent set $X \subset E$ of cardinality $\kappa$ such that for any polynomial $P$ in $n$ variables without constant term and any distinct $x_1,...,x_n \in X$ we have either $P(x_1,...,x_n) \in E$ or $P(x_1,...,x_n)=0$. To prove the strong $\kappa$-algebrability of $E$ we have to find $X \subset E$, $|X|= \kappa$, such that for any non-zero polynomial $P$ without constant term and distinct $x_1,...,x_n \in X$ we have $P(x_1,...,x_n) \in E$.

The notions of $\kappa$-algebrability and strong $\kappa$-algebrability do not coincide. There is a simple example witnessing this: this is the subset $c_{00}$ of $c_0$ consisting of all sequences with real terms equal to zero from some place. It can be proved that $c_{00}$ is algebrable in $c_0$ but is not strongly $1$-algebrable \cite{BG}. 

The notion of algebrability was coined by Gurariy in the early 2000's and first the word 'algebrability' appeared in \cite{AGS} where the research program of finding large algebras inside given sets was posted.
The full definition of algebrability was introduced by Aron, P\'erez-Garc\'ia and Seoane-Sep\'ulveda in \cite{APSS} where it was proved that given a set $E\subset\mathbb{T}$ of measure zero, the set of continuous functions, whose Fourier series expansion is divergent at any point $t\in E$, is dense-algebrable, i.e., there exists an infinite-dimensional, infinitely generated dense subalgebra of $C(\mathbb{T})$, every non-zero element of which has a Fourier series expansion divergent in $E$. Aron and Seoane-Sep\'ulveda proved in \cite{ASS} that the set of everywhere surjective functions from $\C$ to $\C$ (that is, functions $f:\C\to\C$ which take every value $z\in\C$ at each open subset $U\neq\emptyset$ of $\C$, $\omega$-many times) is $\omega$-algebrable.  This was strengthened to $\mathfrak{c}$-algebrability by Aron, Conejero, Peris and Seoane-Sep\'ulveda in \cite{ACPSS}, and to $2^\mathfrak{c}$-algebrability by  Bartoszewicz, G\l \c ab, Pellegrino and Seoane-Sep\'ulveda in \cite{BGPS}. Some other results on algebrability can be found in \cite{GPPSS}, \cite{GPMSS} and \cite{BQ}. 

Most of known results on algebrability are focused on the cases between $\omega$-algebrability and $\mathfrak{c}$-algebrabilty. Recently, $2^\mathfrak{c}$-algebrability was established in $\C^\C$ and $\R^\R$ using independent families of subsets of $\mathfrak{c}$ and a decomposition of $\C$ and $\R$ into $\mathfrak{c}$ copies of Bernstein sets (see \cite{BG}, \cite{BGPS} and \cite{BBG}). 

The notion of strong $\kappa$-algebrability was introduced in \cite{BG}. As it was mentioned above, a strong version of algebrability was proved earlier in \cite{GPPSS}. Also, some strong algebrability results on sequence spaces can be found in \cite{BG}, \cite{BG1} and \cite{BGP}.  

The motivation for this note comes from \cite{BGPS}. There was proved that the set of Siepi\'nski-Zygmund functions is a strongly $\kappa$-algebrable subset of linear algebra $\R^\R$ where $\kappa$ is the cardinality of a family of almost disjoint subsets of $\R$. It was the first time when {\bf strong} $\kappa$-algebrability for $\kappa>\mathfrak{c}$ was proved, since in ZFC there is a family of almost disjoint subsets of $\R$ with cardinality $\mathfrak{c}^+$. The use of a family of almost disjoint subsets of $\R$ was the only known method which gave {\bf strong} $2^\mathfrak{c}$-algebrability of $\R^\R$. Recently G\'amez-Merino and Seoane-Sep\'ulveda noted in \cite{GMSS} that the existence of $2^\mathfrak{c}$ many almost disjoint subset of $\mathfrak{c}$ is independent of ZFC and strong $\kappa$-algebrability of the set of Siepi\'nski-Zygmund functions implies that there is a family of almost disjoint subsets of $\mathfrak{c}$ of cardinality $\kappa$. Thus the strong $2^\mathfrak{c}$-algebrability of the set of Siepi\'nski-Zygmund functions is undecidable in ZFC. Therefore the question arose if one can prove in ZFC that there is a free subalgebra with $2^\mathfrak{c}$ generators of $\R^\R$. Here we answer this question in the affirmative. Moreover, we give a direct method of proving strong $2^\mathfrak{c}$-algebrability of families of strange functions from $\R$ to $\R$. 

In Section 2 we prove a general result stating that, if $X$ is a set of cardinality $\kappa$ where $\kappa^\omega=\kappa$, then $\R^X$ and $\C^X$ contain free linear algebras with $2^\kappa$ generators, that are of the largest possible size. We introduce the notions of an abstract everywhere surjective functions with respect to $\mathcal F$ and we prove maximal strong algebrability of families of such functions. This gives a new insight into nature of the phenomena of algebrability  and lineability of families of surjective functions which was studied by several authors (see \cite{AGS}, \cite{GMMFSSS}, \cite{ACPSS}, \cite{ASS}, \cite{BGPS} and \cite{BBG}).  
In Section 3 we apply the results from Section 2 to show strong algebrability of some families of strange functions in $\R^\R$ and $\C^\C$. Among them there are strongly everywhere surjective and perfectly everywhere surjective functions from $\C$ to $\C$, everywhere discontinuous Darboux functions from $\R$ to $\R$, and real functions whose sets of continuity points equals $K$ for some fixed $G_\delta$ set $K\subsetneq\R$.


The following classical result of Kuratowski and Sierpi\'nski \cite{KS} will be useful in the sequel.
\begin{theorem}[Disjoint Refinement Lemma]
Let $\kappa\geq\omega$. For any sequence $\{P_\alpha:\alpha<\kappa\}$ of (unnecessarily distinct) sets of cardinality $\kappa$ there is a family $\{Q_\alpha:\alpha<\kappa\}$ of sets of cardinality $\kappa$ such that \begin{itemize}
\item[(a)] $\forall \alpha<\kappa\;\;(Q_\alpha\subset P_\alpha);$ 
\item[(b)] $\forall\alpha<\beta<\kappa\;\;(Q_\alpha\cap Q_\beta=\emptyset)$. 
\end{itemize}
\end{theorem}
The family $\{Q_\alpha:\alpha<\kappa\}$ will be called a disjoint refinement of $\{P_\alpha:\alpha<\kappa\}$.

\section{Large free subalgebras of $\C^X$ and $\R^X$}

We say that a family $\{A_s:s\in S\}$ of subsets of $X$ is independent if for any distinct $s_1,s_2,...,s_n\in S$ and any $\varepsilon_i\in\{0,1\}$ for $i=1,...,n$, the set $A^{\varepsilon_1}_{s_1}\cap...\cap A^{\varepsilon_n}_{s_n}$ is nonempty where $B^1=B$ and $B^{0}=X\setminus B$ for $B\subset X$. It is well known that for any set of cardinality $\kappa$ there is an independent family of cardinality $2^\kappa$ (see \cite{BF}).

In the next theorem we prove that the linear algebras $\R^X$ and $\C^X$ contain free linear algebras with $2^{|X|}$ generators, that are of largest possible size. Note that there exist large linear algebras which does not have even free linear subalgebras of one generator, cf. \cite{BBG}, which makes the problem of finding large free subalgebra non trivial. Moreover, as we will show in Section 3, this result can be applied to prove the strong $2^\mathfrak{c}$-algebrability of some families of strange functions in $\R^\R$ and $\C^\C$.


\begin{theorem}\label{main}
Let $X$ be a set of cardinality $\kappa$ where $\kappa=\kappa^\omega$. Let $I$ be a subset of $\R$ (or $\C$) with the nonempty interior. Then there exists a free linear subalgebra of $\R^X$ (or $\C^X$) of $2^\kappa$ generators $\{f_\xi:\xi<2^\kappa\}$ such that $P(f_{\xi_1},\dots,f_{\xi_k})$ maps $X$ onto $P(I^k)$ for every polynomial $P$ in $k$ variables without constant term and any $\xi_1<\xi_2<\dots<\xi_k<2^\kappa$. 
\end{theorem}

\begin{proof}
At first, note that if $\kappa^\omega=\kappa$, then $\kappa\geq\mathfrak{c}$. Therefore $|X|\geq|I|$, so there are surjections from $X$ to $I$. Let $Y=\left([0,1]\times\kappa\right)^\N$ and fix a family $\{A_\xi:\xi<2^\kappa\}$ of independent subsets of $\kappa$. For each $\xi<2^\kappa$, define $\bar{f}_\xi:Y\to[0,1]$ by the formula
$$
\bar{f}_\xi(t_1,y_1,t_2,y_2,...)=\prod_{n=1}^\infty t_n^{\chi_{A_\xi}(y_n)},
$$
where $t_n\in[0,1]$, $y_n\in\kappa$, $\chi_{A_\xi}$ stands for the characteristic function of $A_\xi$, and $0^0=1$. Since $\kappa^\omega=\kappa$ and $\kappa\geq\mathfrak{c}$, then $|Y|=|X|$. Note also that $I$ is of cardinality $\mathfrak{c}$ since it has nonempty interior. Hence we can find two bijections $\phi:X\to Y$ and $\psi:[0,1]\to I$. Then functions $f_\xi=\psi\circ\bar{f}_\xi\circ\phi$, $\xi<2^\kappa$, are free generators in $\R^X$ (or in $\C^X$). Indeed, take $\xi_1<\dots<\xi_k<2^\kappa$ and consider the set
$$
Y_0=\Big\{(t_1,\bar{y}_1,t_2,\bar{y}_1,\dots)\in Y:t_1,\dots,t_k\in[0,1],\; t_i=1/2 \mbox{ for } i>k,\; 
$$
$$
\bar{y}_i\in A_{\xi_i}\setminus\bigcup_{j\neq i}A_{\xi_j} \mbox{ for } i\leq k\mbox{, and }\bar{y}_i\in\bigcap_{j}A_{\xi_j}^c\mbox{ for } i>k\Big\}.
$$
Let $x\in X_0=\phi^{-1}(Y_0)$. Then there are $t_1,\dots,t_k\in[0,1]$ such that $\phi(x)=(t_1,\bar{y}_1,t_2,\bar{y}_1,\dots)\in Y_0$. Let $P$ be a non-zero polynomial in $k$ variables without a constant term. Then
$$
P(f_{\xi_1},\dots,f_{\xi_k})(x)=P(\psi\circ\bar{f}_{\xi_1},\dots,\psi\circ\bar{f}_{\xi_k})(\phi(x))=
P(\psi(t_1),\dots,\psi(t_k)).
$$
Note that for any $z_1,\dots,z_k\in[0,1]$ there is $x\in X_0$ such that 
$$
P(f_{\xi_1},\dots,f_{\xi_k})(x)=P(z_1,\dots,z_k).
$$ 
Since $I$ has the nonempty interior, then $P$ is non-zero on $I^k$. Therefore the function $P(f_{\xi_1},\dots,f_{\xi_k})$ is non-zero on $X_0$, and consequently, it is non-zero on $X$. 

Finally, observe that each function $f_\xi$ is a surjection from $X$ onto $I$. 
\end{proof}

Let us note that the condition $\kappa^\omega=\kappa$ implies that $\kappa\geq\mathfrak{c}$ and $\on{cf}(\kappa)>\omega$. On the other hand, using the well known equality $(\kappa^+)^\lambda=\kappa^+\cdot\kappa^\lambda$ for $\omega\leq\lambda\leq\kappa$, we easily obtain that $\kappa^\omega=\kappa$ provided $\kappa=2^\lambda$ or $\kappa=\omega_n\geq\mathfrak{c}$. 

Let us remark that the function $\bar{f}_\xi$ has a similar definition to that of the function $H_C$ used in the proof of \cite[Proposition 4.2.]{AGS}. However, the functions of type $H_C$ are not algebraically independent for any choices of sets $C$, and therefore they are not appropriate for our purpose.


Now, we will introduce the notion of strongly everywhere surjective function with respect to some family of sets. Our abstract definition generalizes several known notions of everywhere surjective functions. We will prove a general theorem on strong algebrability of the set of strongly everywhere surjective functions with respect to a family of functions. Several corollaries following from this theorem will be presented in the sequel. 

Let $\mathcal F\subset\mathcal{P}(X)$ and $I\subset\C$ (or $\R$). A function $f\in\C^X$ (or in $\R^X$) is $I$-strongly everywhere surjective with respect to $\mathcal F$, in short $f\in\mathcal{SES}(I,\mathcal F)$, if for every set $F\in\mathcal F$ there are $n\in\N$ and a polynomial in $n$ variable such that $f(F)=P(I^n)$, and $|\{x\in F:f(x)=y\}|=|X|$ for any $y\in f(F)$. If $\mathcal F$ consists of all nonempty open sets, then $\mathcal{SES}(\C,\mathcal F)$ is the family of all strongly everywhere surjective complex functions, i.e. the family of functions which map every nonempty open set onto $\C$.  
If $\mathcal F$ consists of all nonempty perfect sets, then $\mathcal{SES}(\C,\mathcal F)$ is the family of all perfectly everywhere surjective complex functions, i.e. the family of functions which map every nonempty perfect set onto $\C$. If $\mathcal F$ consists of all nonempty perfect sets, then $\mathcal{SES}([0,1],\mathcal F)$ consists of real functions $f$ such that $f$ map every uncountable compact set onto a compact set, $f$ is Darboux and the restriction $f\vert_B$ to any uncountable Borel set $B$ is not continuous. Note that functions in $\mathcal{SES}(\R,\mathcal F)$ need not to be surjective, since $P(\R^n)$ not necessary equals $\R$, however, functions from $\mathcal{SES}(\C,\mathcal F)$ are always surjective.
 
\begin{theorem}\label{main1}
Let $X$ be a set of cardinality $\kappa$ where $\kappa=\kappa^\omega$. Let $I$ be a subset of $\C$ (or $\R$) with the nonempty interior. Assume that $|\mathcal F|\leq\kappa$ and $|F|=\kappa$ for every $F\in\mathcal F$. Then the family $\mathcal{SES}(I,\mathcal F)$ is strongly $2^\kappa$-algebrable.
\end{theorem}
\begin{proof}
Let $\{Q_\xi:\xi<\kappa\}$ be a disjoint refinement of a sequence $\mathcal \{F_\xi:\xi<\kappa\}$ in which every element of $\mathcal F$ appears $\kappa$ many times. Adding, if necessary, the elements of the remaining set $X\setminus\bigcup_{\xi<\kappa}Q_\xi$ to $Q_0$ we may assume that $\bigcup_{\xi<\kappa}Q_\xi=X$.   
By Theorem \ref{main}, there is a free algebra $\mathcal A_\xi$ of functions of $2^\kappa$ generators $f_\eta^\xi$, $\eta<2^\kappa$, being surjections from $Q_\xi$ onto $I$, and such that for any polynomial $P$ in $k$ variables without constant term and $\eta_1<\dots<\eta_k<2^\mathfrak{c}$ the function $f:=P(f_{\eta_1}^\xi,\dots,f_{\eta_k}^\xi)$ maps $Q_\xi$ onto $P(I^k)$.

Let $\mathcal A$ be the algebra generated by $\{f_\eta:\eta<2^\kappa\}$ where $f_\eta(x)=f_\eta^\xi(x)$ if $x\in Q_\xi$. Since any $F$ from $\mathcal F$ appears $\kappa$ many times in the sequence $\mathcal \{F_\xi:\xi<\kappa\}$, then $\mathcal A$ is a free linear algebra contained in $\mathcal{SES}(I,\mathcal F)\cup\{0\}$.
\end{proof}


Let $\mathcal F_1$ and $\mathcal F_2$ be two families of subsets of some set $X$. We say that $\mathcal F_1$ is co-initial with $\mathcal F_2$ if for every $F_2\in\mathcal F_2$ there is $F_1\in\mathcal F_1$ such that $F_1\subset F_2$. Note that if $\mathcal F_1$ is co-initial with $\mathcal F_2$, then $\mathcal{SES}(I,\mathcal F_1)\subset\mathcal{SES}(I,\mathcal F_2)$. For example if $X=\C$, $\mathcal F_1$ stands for the family of all nonempty perfect sets, and $\mathcal F_2$ stands for the family of all uncountable Borel sets, then $\mathcal F_1$ and $\mathcal F_2$ are mutually co-initial. Therefore, in this case $\mathcal{SES}(I,\mathcal F_1)=\mathcal{SES}(I,\mathcal F_2)$. 

The next theorem gives us an information on the algebrability of $\mathcal{SES}(I,\mathcal F_1)\setminus\mathcal{SES}(I,\mathcal F_2)$ in the case when $\mathcal F_1$ is not co-initial with $\mathcal F_2$. 

\begin{theorem}\label{main2}
Let $X$ be a set of cardinality $\kappa$ where $\kappa=\kappa^\omega$. Let $I$ be a subset of $\R$ (or $\C$) with the nonempty interior. Assume that $|\mathcal F_i|=\kappa$ and $|F|=\kappa$ for every $F\in\mathcal F_i$, $i=1,2$. Suppose that there is a set $F_2\in\mathcal F_2$ such that $\vert F_1\setminus F_2\vert=\kappa$ for every $F_1\in\mathcal F_1$. 
Then the family $\mathcal{SES}(I,\mathcal F_1)\setminus\mathcal{SES}(I,\mathcal F_2)$ is strongly $2^\kappa$-algebrable.
\end{theorem}
\begin{proof}
Note that the set $X\setminus F_2$, the family $\mathcal F'_1=\{F_1\setminus F_2:F_1\in\mathcal F_1\}$ and every $F\in\mathcal F'_1$ are of cardinality $\kappa$. Therefore by Theorem \ref{main1} there is a free algebra $\mathcal A'$ of $2^\kappa$ generators contained in $\mathcal{SES}(I,\mathcal F'_1)\cup\{0\}$. Let $\mathcal A$ consist of functions $f$ such that $f\vert_{X\setminus F_2}\in\mathcal A'$ and $f\vert_{F_2}=0$. Note that $\mathcal A$ is a free algebra. Clearly, each nonzero $f\in\mathcal A$ is also in $\mathcal{SES}(I,\mathcal F_1)$. Since $I$ has the nonempty interior, then any nonzero polynomial $P$ in $n$ variables is non-constant on $I^n$. Thus $P(I^n)\neq \{0\}$ for every nonzero polynomial $P$ in $n$ variables which means that none $f\in\mathcal A$ is a member of $\mathcal{SES}(I,\mathcal F_2)$.
\end{proof}

\section{Large free subalgebras of $\C^X$ or $\R^X$ consisting of strange functions}


\subsection{Strongly everywhere surjective functions on $\beta\kappa\setminus\kappa$}

Let $\beta\kappa$ be the \v{C}ech-Stone compactification of $\kappa$ equipped with the discrete topology. Then every nonempty open subset of the remainder $\beta\kappa\setminus\kappa$ has the cardinality $2^{2^\kappa}$. Thus the linear algebra $\C^{\beta\kappa\setminus\kappa}$ has $2^{2^{2^\kappa}}$ elements. By Theorem \ref{main} the linear algebra $\C^{\beta\kappa\setminus\kappa}$ contains a free linear algebra of $2^{2^{2^\kappa}}$ generators. Now, we will prove that a free linear algebra $\mathcal A$ of $2^{2^{2^\kappa}}$ generators of $\C^{\beta\kappa\setminus\kappa}$ can be chosen in such a way that $\mathcal A\setminus\{0\}$ consists entirely of strongly everywhere surjective functions (that is strongly everywhere surjective functions with respect to the family of all nonempty open subsets of the remainder $\beta\kappa\setminus\kappa$).

\begin{theorem}
The set of all strongly everywhere surjective functions $f:\beta\kappa\setminus\kappa\to\C$ is strongly $2^{2^{2^\kappa}}$-algebrable.   
\end{theorem}

\begin{proof}
Recall that the basis of $\beta\kappa\setminus\kappa$ consists of sets $U_A=\{p: p$ is a non-principle ultrafilter on $\kappa$ such that $A\in p\}$ for $A\subset\kappa$. Every set $U_A$ is of cardinality $2^{2^\kappa}$. Moreover, there are $2^{2^\kappa}$ distinct non-principal untrafilters on $\kappa$. Now, the result follows from Theorem \ref{main1}.
\end{proof}


\subsection{Perfectly everywhere surjective and strongly everywhere surjective functions from $\C$ to $\C$.}
In this section, we will consider some strange complex (and real) functions. A function $f:\C\to\C$ is called \emph{strongly everywhere surjective} if $f$ takes $\mathfrak{c}$ many times every value $z\in\C$ on every nonempty open subset of $\C$. Recall that a function $f:\C\to\C$ is called \emph{perfectly everywhere surjective} if $f$ maps every perfect subset of $\C$ onto $\C$. We will denote the above classes of functions by $\mathcal{SES}(\C)$ and $\mathcal{PES}(\C)$, respectively. Since every perfect set contains $\mathfrak{c}$ many pairwise disjoint perfect sets, then $\mathcal{SES}\subset\mathcal{PES}$. If $\mathcal F_{\on{perf}}$ and $\mathcal F_{\on{open}}$ stand for the families of all nonempty perfect and nonempty open, respectively, then $\mathcal{SES}(\C,\mathcal F_{\on{perf}})=\mathcal{PES}(\C)$ and $\mathcal{SES}(\C,\mathcal F_{\on{open}})=\mathcal{SES}(\C)$.

The lineability and algebrability of classes of surjective functions were studied by several authors, see \cite{AGS}, \cite{GMMFSSS}, \cite{ACPSS}, \cite{ASS}. In \cite{BGPS} and \cite{BBG} it was proved that $\mathcal{PES}(\C)$ and $\mathcal{SES}(\C)\setminus\mathcal{PES}(\C)$ are $2^\mathfrak{c}$-algebrable. We will show that they are in fact strongly $2^\mathfrak{c}$-algebrable. 

It is known that any function $f:\R\to\R$ can be represented as algebraic sum of two Darboux functions (see for example \cite{Bruckner}). We can prove a similar result for perfectly everywhere surjective functions. 
\begin{proposition}
Any function $f:\C\to\C$ is an algebraic sum of two perfectly everywhere surjective functions. 
\end{proposition} 
\begin{proof}
It was proved in \cite{BGPS} that for every Bernstein set $B\subset\C$ there is $g:B\to\C$ which is perfectly everywhere surjective in the following sense: $g(P\cap B)=\C$ for every perfect set $P$. Let $f:\C\to\C$ and let functions $\tilde{g}_1:B\to\C$, $\tilde{g}_2:\C\setminus B\to\C$ be perfectly everywhere surjective. We extend $\tilde{g}_i$ to $g_i:\C\to\C$ in the following way. Let $g_2(x)=f(x)-\tilde{g}_2(x)$ for $x\in B$ and $g_1(x)=f(x)-\tilde{g}_1(x)$ for $x\in\C\setminus B$. Then $f=g_1+g_2$. 
\end{proof}
The above fact shows that $\mathcal{PES}(\C)$ is far from being linear, since its linear hull contains the whole linear space $\C^\C$. Hence the problem of finding large algebraic structure inside $\mathcal{PES}(\C)\cup\{0\}$ is nontrivial.

By $\mathcal{EDD}(\R)$ we denote the set of all everywhere discontinuous Darboux functions (from $\R$ to $\R$), i.e. nowhere continuous functions which map connected sets onto connected sets.

\begin{theorem}\label{Th2}
The following families of functions are strongly $2^\mathfrak{c}$-algebrable:
\begin{itemize}
\item[(i)] $\mathcal{PES}(\C)$;
\item[(ii)] $\mathcal{SES}(\C)\setminus\mathcal{PES}(\C)$;
\item[(iii)] $\mathcal{EDD}(\R)$. 
\end{itemize}
\end{theorem}

\begin{proof}
At first note that $\mathfrak{c}^\omega=\mathfrak{c}$. 

To see (i) note that there is $\mathfrak{c}$ many nonempty perfect sets of cardinality $\mathfrak{c}$. Thus (i) follows from Theorem \ref{main1}.

To prove (ii) fix a perfect nowhere dense set $P\subset\C$. Note that, for any nonempty open set $U$, the set $U\setminus P$ is nonempty and open. Thus, if $\mathcal F_1$ is a family of all nonempty open subsets of $\C$ and $\mathcal F_2$ is a family of all nonempty perfect subsets of $\C$, then by Theorem \ref{main2} we obtain (ii).

To prove (iii) one needs to repeat the construction given in the proof of part (i), but instead of a complex domain one should consider a real domain. 
\end{proof}


\subsection{Everywhere surjective functions with respect to $\mathcal F$.} In this section we will apply Theorem \ref{main2} to sharpen Theorem \ref{Th2}(ii). Let  $\mathcal F_{\on{Leb}}$ and $\mathcal F_{\on{Baire}}$ stand for families of all  measurable set of positive Lebesgue measure and all sets with the Baire property of second category, respectively. We start from the following observation about functions from classes $\mathcal{SES}(\C,\mathcal F_{\on{Leb}})$, $\mathcal{SES}(\C,\mathcal F_{\on{Baire}})$ and $\mathcal{PES}(\C)$. 

\begin{proposition}\begin{itemize}
\item[(i)] Let $f\in\mathcal{SES}(\C,\mathcal F_{\on{Leb}})$. Then for any measurable  set $A$ of positive measure the restriction $f\vert_A$ is non-measurable.  
\item[(ii)] Let $f\in\mathcal{SES}(\C,\mathcal F_{\on{Baire}})$. Then for any set $A$ with the Baire property  of second category, the restriction $f\vert_A$ has no Baire property.  
\item[(iii)] Let $f\in\mathcal{PES}(\C)=\mathcal{SES}(\C,\mathcal F_{\on{perf}})$. Then for any uncountable Borel set $A$, the restriction $f\vert_A$ is non-Borel. Moreover the preimages $f^{-1}(\{x\})$ of all singletons are Bernstein sets. 
\end{itemize}
\end{proposition}

\begin{proof}
(i) Let $A$ be a measurable set of positive measure. Suppose that $f\vert_A$ is measurable. Then $A$ contains a compact set $C$ of positive measure such that $f\vert_C$ is continuous. Since $f\in\mathcal{SES}(\C,\mathcal F_{\on{Leb}})$, then $f(C)=\C$ which contradicts continuity of $f\vert_C$.  

(ii) Now, let $A$ be a set with the Baire property of second category and suppose that $f\vert_A$ has the Baire property. Then there is a $G_\delta$ subset $C$ of $A$ such that $f\vert_C$ is continuous and $A\setminus C$ is meager. Put $U=\{x\in\R:(x-\varepsilon,x+\varepsilon)\cap C$ is meager$\}$. Then $U$ is an open set and $C\cap U$ is meager. Let $n\in\N$ and $x\in C\setminus U$. Since $f\in\mathcal{SES}(\C,\mathcal F_{\on{Baire}})$, then $f((x-1/n,x+1/n))=\C$ which contradicts the fact that $f\vert_C$ is continuous. 

(iii) Let $A$ be an uncountable Borel set and suppose that $f\vert_A$ is Borel. Then there is a nonempty perfect set $C\subset A$ such that $f\vert_C$ is continuous. But $f(C)=\C$. A contradiction. 

Note that, if $f\in\mathcal{PES}(\C)$, then $f^{-1}(\{x\})$ and $\C\setminus f^{-1}(\{x\})$ have nonempty intersections with any perfect subset of $\C$. Hence $f^{-1}(\{x\})$ is a Bernstein set. 
\end{proof}

We have the following inclusions between the considered classes of functions:  
$\mathcal{PES}(\C)\subset \mathcal{SES}(\C,\mathcal F_{\on{Leb}})\cap \mathcal{SES}(\C,\mathcal F_{\on{Baire}})$, $\mathcal{SES}(\C,\mathcal F_{\on{Leb}})\subset\mathcal{SES}(\C)$ and $\mathcal{SES}(\C,\mathcal F_{\on{Baire}})\subset\mathcal{SES}(\C)$. Moreover there are no inclusions between $\mathcal{SES}(\C,\mathcal F_{\on{Leb}})$ and $\mathcal{SES}(\C,\mathcal F_{\on{Baire}})$. 
\begin{theorem}\label{non-initial}
The following families of function from $\C$ to $\C$ are strongly $2^\mathfrak{c}$-algebrable: 
\begin{itemize}
\item[(i)] $\mathcal{SES}(\C,\mathcal F_{\on{Leb}})\setminus\mathcal{PES}(\C)$ and $\mathcal{SES}(\C,\mathcal F_{\on{Baire}})\setminus\mathcal{PES}(\C)$;
\item[(ii)]  $\mathcal{SES}(\C,\mathcal F_{\on{Baire}})\setminus\mathcal{SES}(\C,\mathcal F_{\on{Leb}})$ and $\mathcal{SES}(\C,\mathcal F_{\on{Leb}})\setminus\mathcal{SES}(\C,\mathcal F_{\on{Baire}})$;  
\item[(iii)] $\mathcal{SES}(\C)\setminus\mathcal{SES}(\C,\mathcal F_{\on{Leb}})$ and $\mathcal{SES}(\C)\setminus\mathcal{SES}(\C,\mathcal F_{\on{Baire}})$.    
\end{itemize}
\end{theorem}

\begin{proof}
We need only to prove that the respective families of sets fulfill the assumptions of Theorem \ref{main2}. To prove (i) note that the ternary Cantor set witnesses that $\mathcal F_{\on{Leb}}$ and $\mathcal F_{\on{Baire}}$ are not co-initial with $\mathcal F_{\on{perf}}$. To prove (ii) let $A$ and $B$ be such that $A\cap B=\emptyset$, $A\cup B=\C$, $A$ is null and $B$ is meager. Then $A$ witnesses that $\mathcal F_{\on{Leb}}$ is not co-initial with $\mathcal F_{\on{Baire}}$ while $B$ witnesses that $\mathcal F_{\on{Baire}}$ is not co-initial with $\mathcal F_{\on{Leb}}$. To see (iii) note that neither $A$ nor $B$ have nonempty interiors. Thus $A$ witnesses that $\mathcal F_{\on{open}}$ is not co-initial with $\mathcal F_{\on{Baire}}$ while $B$ witnesses that $\mathcal F_{\on{open}}$ is not co-initial with $\mathcal F_{\on{Leb}}$.   
\end{proof}


\subsection{Measurable strongly everywhere surjective functions.}

A function $f:\R\to\R$ is called \emph{Sierpi\'nski function} if any perfect set $P\subset\R$ contains a perfect set $Q\subset P$ such that $f\vert_Q$ is continuous. Such functions were first considered by Sierpi\'nski in \cite{S}. Marczewski proved in \cite{M} that Sierpi\'nski functions are exactly  measurable functions with respect to the Marczewski $\sigma$-algebra $s(\R)$ given by
$$
s(\R)=\{A\subset\R: \forall P\in\mathcal{P}\exists Q\subset P,Q\in\mathcal{P}(Q\subset A\cap P\mbox{ or }Q\subset P\setminus A)\},
$$ 
where $\mathcal{P}$ stands for the family of all nonempty perfect subsets of reals. The family of hereditary subsets of $s(\R)$ forms a $\sigma$-ideal of the form
$$
s_0(\R)=\{A\subset\R: \forall P\in\mathcal{P}\exists Q\subset P,Q\in\mathcal{P}(Q\subset P\setminus A)\}.
$$
One can consider $s(X)$ and $s_0(X)$ for any Polish space $X$. Then $s(X)$ and $s_0(X)$ are a $\sigma$-algebra and a $\sigma$-ideal of subsets of $X$.    

It is well-known that Bernstein sets are not measurable and do not have the Baire property. Also they do not belong to $s(\C)$. Thus perfectly everywhere surjective functions are not $s(\C)$-measurable.  The following result sharpens Theorem \ref{non-initial}(iii) and consequently, it is a strengthening of Theorem \ref{Th2}(ii).

\begin{theorem}
By $\mathcal G$ denote the set of all functions $f:\C\to\C$ which fulfill the following conditions: 
\begin{itemize}
\item[(a)] $f$ is strongly everywhere surjective;
\item[(b)] for any perfect set $S$ there is a perfect set $S'\subset S$ with $f\vert_{S'}=0$, in particular $f$ is $s(\C)$-measurable;
\item[(c)] $f$ is Lebesgue measurable and $f$ has the Baire property. 
\end{itemize}
Then $\mathcal{G}$ is strongly $2^\mathfrak{c}$-algebrable. 
\end{theorem}

\begin{proof}
Let $\{V_n:n\in\N\}$ be a basis of open sets in $\C$. Consider a family $\{C_n:n\in\N\}$ of pairwise disjoint nowhere dense perfect and Lebesgue null sets such that $C_n\subset V_n$ for $n\in\N$. For any $n\in\N$ fix an $s_0(\C)$ set $D_n\subset C_n$ of cardinality $\mathfrak{c}$. Let $\mathcal F=\{D_{n}:n\in\N\}$ and let $\mathcal A'\subset\mathcal{SES}(\C,\mathcal F)$ be a free linear algebra of surjective functions from $D:=\bigcup_{n\in\N}D_n$ onto $\C$ generated by $f'_{\xi}$, $\xi<2^\mathfrak{c}$. Let $\mathcal A$  be the free  linear algebra generated by functions
$$
f_\xi(x)=
\left\{\begin{array}{ll}
f'_\xi(x), & \textrm{ if }x\in D,\\
0, & \textrm{ if }x\notin D.
\end{array} \right.
$$
Since $s_0(\C)$ is a $\sigma$-ideal, then any perfect set $S$ has a perfect subset $S'$ disjoint from every $D_n$. Moreover, functions from $\mathcal A$ takes every value $\mathfrak{c}$ many times on $D_n$.  
\end{proof}

Let $\mathcal{G}_1$ be the family of all functions from $\mathcal{SES}(\C)$ which have the Baire property but are neither measurable nor $s$-measurable. Let $\mathcal{G}_2$ be the family of all functions from $\mathcal{SES}(\C)$ which are measurable but neither have the Baire property nor are $s(\C)$-measurable.

\begin{theorem}
The families $\mathcal G_1$ and $\mathcal G_2$ are strongly $2^\mathfrak{c}$-algebrable. 
\end{theorem}

\begin{proof}
Let $A,B\subset\C$ be such that $A\cap B=\emptyset$, $A\cup B=\C$, $A$ is of the first category and $B$ is null. 

Enumerate all perfect subsets of $A$ as $\mathcal F=\{P_\alpha:\alpha<\mathfrak{c}\}$, and let $\mathcal A\subset\mathcal{SES}(\C,\mathcal F)$. Therefore, for every $f\in\mathcal A$, the preimages $f^{-1}(\{z\})$, $z\in\C$, are Bernstein sets in $A$. Hence $f$ is not $s(\C)$-measurable. Since $A$ is of full measure, there is no set of positive measure on which $f$ is continuous. Thus $f$ is non-measurable. For any open nonempty set $U\subset\C$, there is a perfect set $P\subset U\cap A$. Since every perfect set has $\mathfrak{c}$ many pairwise disjoint perfect subsets $P_\alpha\subset P$, $\alpha<\mathfrak{c}$, and $f$ maps each $P_\alpha$ onto $\C$, then $f$ is strongly everywhere surjective. Finally, since $f$ is constant on a comeager set $B$, then it has the Baire property. 

To prove the remaining part of the assertion, replace $A$ by $B$ in the above proof.
\end{proof}

We end our considerations on strongly everywhere surjective functions with respect to family $\mathcal F$ with remark on everywhere surjective functions. 
If in the definition of $\mathcal{SES}(\mathcal F)$ we do not assume that nonempty preimages of singletons have maximal cardinality, that is $|\{x\in F:f(x)=y\}|=|X|$ for any $y\in f(F)$ and $F\in\mathcal F$, then we will obtain the class $\mathcal{ES}(\mathcal F)$ of everywhere surjective functions with respect to $\mathcal F$. Note that $\mathcal{ES}(\C,\mathcal F_{\on{perf}})=\mathcal{SES}(\C,\mathcal F_{\on{perf}})=\mathcal{PES}$ since any perfect set $P$ contains $\mathfrak{c}$ many pairwise disjoint nonempty perfect subsets. On the other hand, $\mathcal{SES}=\mathcal{SES}(\C,\mathcal F_{\on{open}})\subsetneq\mathcal{ES}(\C,\mathcal F_{\on{open}})$. Moreover $\mathcal{ES}\setminus\mathcal{SES}$ is $\mathfrak{c}$-lineable \cite{GMMFSSS}. It is an open problem if $\mathcal{ES}\setminus\mathcal{SES}$ is $2^\mathfrak{c}$ -lineable. In general, one can ask if $\mathcal{ES}(\mathcal F)\setminus\mathcal{SES}(\mathcal F)$ is $2^\mathfrak{c}$-lineable or (strongly) $2^\mathfrak{c}$-algebrable under some reasonable assumptions on $\mathcal F$.


\subsection{Functions whose sets of continuity points is a fixed $G_\delta$ set.}

It is well-known that the set of continuity points of an arbitrary function $f:\R\to\R$ is a $G_\delta$-set. Several authors considered lineability and coneability of families of functions with prescribed sets of discontinuity points. Garc\'ia-Pacheco, Palmberg and Seoane-Sep\'ulveda in \cite[Theorem 5.1]{GPPSS} proved the $\omega$-lineability of the set of functions with finitely many points of continuity. 
Aizpuru,  P\'erez-Eslava, Garc\'ia-Pacheco and Seoane-Sep\'ulveda established in \cite{APGS} that the set of all functions $f:\R \rightarrow \R$ which are continuous only at the points of a fixed open set $U$ (a fixed $G_\delta$ set, respectively) is lineable (coneable). Recently, Bartoszewicz, Bienias and G\l\c{a}b established $2^\mathfrak{c}$-algebrability of the family of functions whose sets of continuity points equal to $K$ for a fixed closed set $K\subsetneq\R$ (or $K\subsetneq\C$) \cite{BBG}. In Theorem \ref{theorem_algebrability}, we will give a condition equivalent for strong $2^\mathfrak{c}$-algebrability  of the family of functions whose sets of continuity points equal to a fixed $G_\delta$ set. 

All the sets below will be subsets of $\R$. We say that a nonempty set $A$ is $\mathfrak{c}$-dense in itself if either $A\cap I=\emptyset$ or $\vert A\cap I\vert=\mathfrak{c}$ for every open interval $I$. Note that any $\mathfrak{c}$-dense in itself set is dense in itself. Observe that a nonempty set $A$ is dense in itself if either $A\cap I=\emptyset$ or $\vert A\cap I\vert\geq\omega$ for every open interval $I$. This shows that being $\mathfrak{c}$-dense in itself is a natural strengthening of being dense in itself. However these two notions coincide in the class of closed sets with the class of nonempty perfect sets. It follows directly from the definition of $\mathfrak{c}$-dense in itself set that if $A$ is $\mathfrak{c}$-dense in itself and $C$ is closed, then $A\setminus C$ is $\mathfrak{c}$-dense in itself.

\begin{lemma}\label{isolated_points}
Let $F$ be a Borel $\mathfrak{c}$-dense in itself set. Suppose that $x\in F$ and $I$ is an open neighbourhood of $x$. Then there is a nonempty perfect set $H\subset F\cap I$ containing $x$.  
\end{lemma}

\begin{proof}
Since $F$ is dense in itself, there is a sequence $(x_n)$ tending to $x$. We can find pairwise disjoint open intervals $I_n$ with $x_n\in I_n$, $n\in\N$. Since $F$ is $\mathfrak{c}$-dense in itself, then $|F\cap I_n|=\mathfrak{c}$. Hence the set  $F\cap I_n$ contains a subset $H_n$ homeomorphic to the Cantor set.  Let $H=\{x\}\cup\bigcup_{n} H_n$. Then $H$ is perfect. 
\end{proof}

\begin{lemma}\label{F_sigma-c-dense}
Let $F$ be a non-closed $\mathfrak{c}$-dense in itself $F_\sigma$ set. Then there is a sequence $C_1\subsetneq C_2\subsetneq C_3\subsetneq\dots$ of perfect sets such that $F=\bigcup_{n}C_n$. 
\end{lemma}

\begin{proof}
Let $D_1\subset D_2\subset\dots$ be a sequence of closed sets with $F=\bigcup_{n}D_n$. Let $A=\{x_k:k\in J\}$ be the set of all isolated points of $D_1$ where $J\subset\N$. One can find a family $\{I_k\}$ of pairwise disjoint open intervals such that $x_k\in I_k$ and $\mathrm{diam}(I_k)\leq 1/k$. By Lemma \ref{isolated_points} we can find closed $\mathfrak{c}$-dense in itself sets $H_k\subset I_k\cap F$ with $x_k\in H_k$, $k\in J$. Put $C_1=D_1\cup\bigcup_k H_k$. Clearly $C_1\subset F$. 

We will prove that $C_1$ is perfect. The construction ensures that $C$ is dense in itself. Thus it is enough to show that $C_1$ is closed. Let $(c_n)$ be a sequence of elements of $C_1$ tending to $c$. Consider the following cases.\\
{\bf 1.} If infinitely many elements of $(c_n)$ is from $D_1$, then $c\in D_1$.\\
{\bf 2.} Assume that every element of $(c_n)$ is in $\bigcup_k H_k$. For every $n$ there is $k_n$ with $c_n\in H_{k_n}$.\\ 
{\bf 2(a).} If $k_n=n_0$ for infinitely many $n$'s, then $c\in H_{n_0}\subset C_1$.\\ 
{\bf 2(b).} Assume that $(k_n)$ is a one-to-one sequence of natural numbers. Since $\textrm{diam}(H_{k_n})\leq 1/k_n$, then $d(x_{k_n},c)\leq d(c_n,x_{k_n})+ d(c_n,c)\leq 1/k_n+d(c_n,c)\to 0$. Therefore $c\in D_1$. 

Since $F$ is not closed, then $F\setminus C_1$ is nonempty   $\mathfrak{c}$-dense in itself $F_\sigma$ set. Using the above reasoning we can find a closed set $C_2'$ inside $F$ such that  $C_2'$ is $\mathfrak{c}$-dense in itself and $C_2'$ contains $D_2$. Put $C_2=C_1\cup C_2'$. Then $C_2$ is perfect and the difference $C_2\setminus C_1$ is $\mathfrak{c}$-dense in itself. Proceeding inductively we define $C_n$. Since $C_n$ contains $D_n$, the result follows. 
\end{proof}

For a $G_\delta$ set $G$, by $\mathcal C_G$ we denote the set of all functions $f:\R\to\R$ such that the set of continuity points of $f$ equals $G$. 

\begin{theorem}\label{theorem_algebrability}
Let $G\subset\R$ be a $G_\delta$-set. The following conditions are equivalent:
\begin{itemize}
\item[(i)] $\mathcal C_G$ is strongly $2^\mathfrak{c}$-algebrable;
\item[(ii)] $\mathcal C_G$ is $\mathfrak{c}^+$-lineable; 
\item[(iii)] $\R\setminus G$ is $\mathfrak{c}$-dense in itself.
\end{itemize}
\end{theorem}

\begin{proof}
The implication (i)$\Rightarrow$(ii) is obvious. 

(ii)$\Rightarrow$(iii): Assume that $\R\setminus G$ is not $\mathfrak{c}$-dense in itself. Then there is an open interval $I\subset\R$ such that $0<|I\setminus G|\leq\omega$. Let $I\setminus G=\{a_j:j\in J\}$ where $J\subset\N$. Note that every function $f:I\to\R$ such that $G\cap I$ is a set of all continuity points of $f$ can be written as $f=f\vert_{I\setminus G}\cup f\vert_{G\cap I}$ where $f\vert_{G\cap I}$ is continuous. There are exactly $\mathfrak{c}$ many continuous functions from $G\cap I$ to $\R$ and at least $\mathfrak{c}$ many functions from $I\setminus G$ to $\R$. Hence there are $\mathfrak{c}$ many functions $f:I\to\R$ such that $G\cap I$ is a set of all continuity points of $f$. 

Suppose that $\mathcal C_G$ is $\mathfrak{c}^+$-lineable. Let $\{f_\alpha:\alpha<\mathfrak{c}^+\}$ be a basis of linear space $M\subset \mathcal C_G\cup\{0\}$. By the above observation there are $\alpha<\beta<\mathfrak{c}^+$ such that $f_\alpha\vert_I=f_\beta\vert_I$. Hence $f_\alpha-f_\beta$ equals zero on $I$. But we have assumed that $I\setminus G\neq\emptyset$. Thus the set of continuity points of $f_\alpha-f_\beta$ is not equal to $G$ and we reach a contradiction. 

(iii)$\Rightarrow$(i): Let $\R\setminus G$ be perfect. For any open set $U$, the intersection $U\cap(\R\setminus G)$ is either empty or of cardinality continuum. Let $\mathcal F=\{U\cap(\R\setminus G):U$ is open and $U\cap(\R\setminus G)\neq\emptyset\}$ and let $\mathcal A'\subset\mathcal{SES}(\R,\mathcal F)$ be the free linear algebra generated by surjections $f'_{\xi}:\R\setminus G\to\R$, $\xi<2^\mathfrak{c}$. Finally we define functions $f_\xi$, $\xi<2^\mathfrak{c}$ by 
$$
f_\xi(x)=
\left\{\begin{array}{ll}
f'_{\xi}(x), & \textrm{ if }x\in \R\setminus G,\\
0, & \textrm{ if }x\in G.
\end{array} \right.
$$
Clearly the functions $f_\xi$, $\xi<2^\mathfrak{c}$, generate the desired free linear algebra.

Now, assume that $\R\setminus G$ is not closed.
Since $\R\setminus G$ is $\mathfrak{c}$-dense in itself, then by Lemma \ref{F_sigma-c-dense} there is a sequence $C_n$ of perfect sets such that $C_1$ and $C_{n+1}\setminus C_n$ are $\mathfrak{c}$-dense in itself, and $\R\setminus G=\bigcup_{n}C_n$. Then $G=\bigcap_{n}\R\setminus C_n$. Let $\mathcal F_n=\{U\cap(C_n\setminus C_{n-1}):U$ is open and $U\cap(C_n\setminus C_{n-1})\neq\emptyset\}$ where $C_0=\emptyset$. Let $\mathcal A_n\subset\mathcal{SES}([0,1/n],\mathcal F_n)$ denote the free linear algebra generated by functions $f_{\xi}^n:C_n\setminus C_{n-1}\to[0,1/n]$, $\xi<2^\mathfrak{c}$. Finally, we define functions $f_\xi$, $\xi<2^\mathfrak{c}$ by 
$$
f_\xi(x)=
\left\{\begin{array}{ll}
f_{\xi}^n(x), & \textrm{ if }x\in C_n\setminus C_{n-1},\\
0, & \textrm{ if }x\in G.
\end{array} \right.
$$
We will prove that the functions $f_\xi$, $\xi<2^\mathfrak{c}$, generate the desired free linear algebra. Let $P$ be a polynomial in $m$ variables. Since $f:=P(f_{\xi_1},\dots f_{\xi_m})\vert_{C_n\setminus C_{n-1}}\in\mathcal{A}_n$, then $f$ maps every nonempty open subset of $C_n\setminus C_{n-1}$ onto $P([0,1/n]^m)$, and therefore $f$ is discontinuous at each point of $C_n\setminus C_{n-1}$. 

It remains to prove that $f$ continuous at any point of $G$. Let $x\in G$. If $x$ is an interior point of $G$, then $f$ equals zero on some neighborhood of $x$, and consequently $f$ is continuous at $x$. Assume that $x$ is not interior point of $G$ and $x_k$ is a sequence of elements from $\R\setminus G$ tending to $x$. Since $x_k\notin G$, then there is $k_n$ with $x\in C_{k_n}$. Clearly $k_n\to\infty$. Therefore $f(x_k)\in P([0,1/k_n]^m)$. Hence $f(x_k)\to 0=f(x)$.    
\end{proof}

Theorem \ref{theorem_algebrability} says that if $\R\setminus G$ is not $\mathfrak{c}$-dense in itself, then $\mathcal C_G$ is not $\mathfrak{c}^+$-lineable. A natural question arises: Is $\mathcal C_G$ $\mathfrak{c}$-lineable? The next theorem gives the affirmative answer.

\begin{theorem}\label{c+lineability}
Let $G\subset\R$ be a $G_\delta$ set such that $\R\setminus G$ is not $\mathfrak{c}$-dense in itself. Then $\mathcal C_G$ is $\mathfrak{c}$-lineable. 
\end{theorem}

\begin{proof}
Since $\R\setminus G$ is not $\mathfrak{c}$-dense in itself, there is open set $I$ with $1\leq \vert(\R\setminus G)\cap I\vert\leq\omega$. Let $U=\bigcup\{I:I$ is open and $1\leq\vert(\R\setminus G)\cap I\vert\leq\omega\}$. Clearly $U$ is a nonempty open set. Then $G\setminus U$ is of type $G_\delta$ and $(\R\setminus U)\setminus(G\setminus U)$ is $\mathfrak{c}$-dense in itself. Let $\{I_n\}$ be the family of all connected components of $U$. Fix $I_n$. Let $\{a_i:i\in N_n\}$ be an enumeration of the set $(\R\setminus G)\cap I_n$ where $N_n\subset\N$. 

Let $\{r_\alpha:\alpha<\mathfrak{c}\}\cup\{\pi\}$ be a set of positive real numbers which is linearly independent over $\Q$. 
For any $a_i$ let $\varepsilon_{\alpha,i}$ be a positive real number such that $(a_i-2\varepsilon_{\alpha,i},a_i+2\varepsilon_{\alpha,i})\subset I_n$ and $\sin(r_{\alpha}/\varepsilon_{\alpha,i})=0$. Define $\psi_{\alpha,i}: I_n\to\R$ in the following way. Put $\psi_{\alpha,i}(t)=\sin(\frac{r_\alpha}{t-a_i})$ if $t\in(a_i-\varepsilon_{\alpha,i},a_i)\cup(a_i,a_i+\varepsilon_{\alpha,i})$, and $\psi_{\alpha,i}(t)=0$ otherwise. Define $\varphi_{\alpha, n}:I_n\to\R$ by the formula
$$
\varphi_{\alpha,n}(t)=\vert I_n\vert\sum_{i\in N_n}\frac{1}{2^i}\psi_{\alpha,i}(t).
$$ 
By $\varphi_{\alpha}':\R\setminus U\to\R$ denote algebraically independent functions such that the linear algebra generated by $\{\varphi_\alpha':\alpha<\mathfrak{c}\}$ consists of functions $f:\R\to\R$ which set of continuity points equal $G\cup U$ and $f\vert_{G\cup U}=0$. It can be done using the same reasoning as in the proof of Theorem \ref{theorem_algebrability}. Finally, we define $\varphi_\alpha:\R\to\R$ in the following way. Put $\varphi_\alpha(t)=\varphi_{\alpha,n}(t)$ if $t\in I_n$ and $\varphi_\alpha(t)=\varphi_\alpha '(t)$ if $t\in\R\setminus U$. We will prove that the family of functions $\{\varphi_\alpha:\alpha<\mathfrak{c}\}$ is linearly independent and any nonzero linear combination of its elements is in $\mathcal{C}_G$.  

Let $c_1,\dots,c_k$ be a nonzero real numbers, $\alpha_1<\dots<\alpha_k<\mathfrak{c}$, and consider $\varphi:=c_1\varphi_{\alpha_1}+\dots+c_k\varphi_{\alpha_k}$. We need to prove that $\varphi$ belongs to $\mathcal C_G$. 
\begin{claim}
The set of continuity points of $\varphi\vert_{I_n}=c_1\varphi_{\alpha_1,n}+\dots+c_k\varphi_{\alpha_k,n}$ equals $G\cap I_n$
\end{claim} 
\begin{proof}
We have
$$
\varphi\vert_{I_n}(t)=c_1\varphi_{\alpha_1,n}(t)+\dots+c_k\varphi_{\alpha_k,n}(t)=\vert I_n\vert\sum_{i\in N_n}\frac{1}{2^i}(c_1\psi_{\alpha_1,i}(t)+\dots+c_k\psi_{\alpha_k,i}(t))
$$
Using the multidimensional Kronecker Lemma (see \cite[Theorem 442]{HW}) we infer that $\{(\psi_{\alpha_1,i}(t),\dots,\psi_{\alpha_k,i}(t)):t\in (a_i-\delta,a_i+\delta)\}=[-1,1]^k$ for any $\delta>0$. Therefore, the range of the function $t\mapsto c_1\psi_{\alpha_1,i}(t)+\dots+c_k\psi_{\alpha_k,i}(t)$ restricted to any neighborhood of $a_i$ equals the range of the function $[-1,1]^k\ni (t_1,\dots,t_k)\mapsto c_1t_1+\dots+c_kt_k$. Thus $t\mapsto c_1\psi_{\alpha_1,i}(t)+\dots+c_k\psi_{\alpha_k,i}(t)$ is discontinuous at $a_i$. By the Weierstrass M-test, the function $\varphi\vert_{I_n}$ is continuous at every point of $G\cap I_n$ and the function $\varphi\vert_{I_n}-(c_1\psi_{\alpha_1,i}+\dots+c_k\psi_{\alpha_k,i})$ is continuous at every point of $(G\cap I_n)\cup\{a_i\}$. Therefore $\varphi\vert_{I_n}$ is discontinuous at any $a_i$. 
\end{proof}

\begin{claim}
Let $x\in\R\setminus U$. Then $x$ is continuity point of $\varphi$ if and only if $x\in G$. 
\end{claim}

\begin{proof}
Assume first that $x\notin G$. Since $(\R\setminus U)\setminus(G\setminus U)$ is $\mathfrak{c}$-dense in itself, then by the construction the function $\varphi\vert_{\R\setminus G}=c_1\varphi_{\alpha_1}'+\dots+c_k\varphi_{\alpha_k}'$ is discontinuous at $x$. Thus $\varphi$ is also discontinuous at $x$.

Assume now that $x\in G$. Then $\varphi(x)=0$ and $\varphi\vert_{\R\setminus G}$ is continuous at $x$. If $x$ is an interior point of $\R\setminus U$, then $\varphi$ is continuous at $x$. Suppose then that $x$ a boundary point of $G$. There is a sequence $(x_m)$ of elements of $U$ tending to $x$. Let $n_m$ be such that $x_m\in I_{n_m}$. 

If $(n_m)$ is a bounded sequence, we may assume that $n_m=n'$ for every $m$. Then $x$ is an endpoint of the interval $I_{n'}=(a,b)$. We may assume that $x=a$ and $(x_m)$ is decreasing. If $N_{n'}$ is finite, then let $a_{i'}$ be the smallest element of the set $\{a_i:i\in N_{n'}\}$. Note that $\psi_{\alpha_j,i}(t)=0$ for $t\in(a,a_{i'}-\varepsilon_{\alpha_j,i'})$, $j=1,\dots,k$ and $i\in N_{n'}$. Hence $\varphi(x_m)=0$ for large enough $m$. If $N_{n'}$ is infinite, then fix $\delta>0$. Let $i'$ be such that $\vert I_{n'}\vert(\vert c_1\vert+\dots\vert c_k\vert)\sum_{i=i'}^\infty 1/2^i<\delta$. Let $t\in(a,\min_{i<i',j=1,\dots,k}(a_i-\varepsilon_{\alpha_j,i}))$. By the construction $\vert\varphi(t)\vert<\delta$. Therefore $\varphi(x_m)\to 0$.  

If $(n_m)$ is an unbounded sequence, then we may assume that it is one-to-one. Since $x_m\in I_{n_m}$ and $x_m\to x$, the diameters of $I_{n_m}$ tends to zero with $m$ tending to $\infty$. By the construction $\vert\varphi(x_m)\vert\leq(\vert c_1\vert+\dots+\vert c_k\vert)\vert I_{n_m}\vert$. Hence $\varphi(x_m)\to 0$.
\end{proof}
\end{proof}

Assume that $G\subset\R$ is a $G_\delta$ set such that $\R\setminus G$ is not $\mathfrak{c}$-dense in itself. Is $\mathcal C_G$ strongly $\mathfrak{c}$-algebrable? We are not able to prove it but we show strong $\mathfrak{c}$-algebrability of $\mathcal{C}_G$ in some special case. Let $I$ be, as in the proof of Theorem \ref{c+lineability}, an open set with $1\leq \vert(\R\setminus G)\cap I\vert\leq\omega$. Let $U=\bigcup\{I:I$ is open and $1\leq\vert(\R\setminus G)\cap I\vert\leq\omega\}$.

\begin{theorem}
Assume that $U\setminus G$ is nonempty and discrete. Then $\mathcal C_G$ is strongly $\mathfrak{c}$-algebrable. 
\end{theorem}

\begin{proof}
The construction will be very similar to that from the proof of Theorem \ref{c+lineability} and therefore we will omit some details. 

Let $\{a_i:i\in N\}$ be an enumeration of the set $U\setminus G$ and let $\{I_i:i\in N\}$ be a family of pairwise disjoint intervals with $a_i\in I_i$. Let $\{r_\alpha:\alpha<\mathfrak{c}\}\cup\{\pi\}$ be a set of positive real numbers which is linearly independent over $\Q$. 
For any $a_i$ let $\varepsilon_{\alpha,i}$ be a positive real number such that $(a_i-2\varepsilon_{\alpha,i},a_i+2\varepsilon_{\alpha,i})\subset I_i$ and $\sin(r_{\alpha}/\varepsilon_{\alpha,i})=0$. Define $\psi_{\alpha,i}: U\to\R$ in the following way. Put $\psi_{\alpha,i}(t)=\varepsilon_{\alpha,i}\sin(\frac{r_\alpha}{t-a_i})$ if $t\in(a_i-\varepsilon_{\alpha,i},a_i)\cup(a_i,a_i+\varepsilon_{\alpha,i})$, and $\psi_{\alpha,i}(t)=0$ otherwise. Define $\varphi_{\alpha}':U\to\R$ by the formula
$$
\varphi_{\alpha}'(t)=\sum_{i\in N}\psi_{\alpha,i}(t).
$$ 
Note that $\varphi_\alpha'$ is well-defined since $\psi_{\alpha,i}$, $i\in N$, have distinct supports. 
By $\varphi_{\alpha}'':\R\setminus U\to\R$ denote algebraically independent functions such that a linear algebra generated by $\{\varphi_\alpha'':\alpha<\mathfrak{c}\}$ consists of functions $f:\R\to\R$ which set of continuity points equal $G\cup U$ and $f\vert_{G\cup U}=0$. Finally we define $\varphi_\alpha:\R\to\R$ in the following way. Put $\varphi_\alpha(t)=\varphi_{\alpha}'(t)$ if $t\in U$ and $\varphi_\alpha(t)=\varphi_\alpha ''(t)$ if $t\in\R\setminus U$. We will prove that the family of functions $\{\varphi_\alpha:\alpha<\mathfrak{c}\}$ is algebraically independent and any nonzero algebraic combination of its elements is in $\mathcal{C}_G$.  

Let $P$ be a nonzero polynomial in $k$ variables and let $\alpha_1<\dots<\alpha_k<\mathfrak{c}$. Consider the function $\varphi=P(\varphi_{\alpha_1},\dots\varphi_{\alpha_k})$. Since $\{(\psi_{\alpha_1,i}(t),\dots,\psi_{\alpha_k,i}(t)):t\in (a_i-\delta,a_i+\delta)\}=[-1,1]^k$ for any $\delta>0$. Therefore the range of the function $t\mapsto P(\psi_{\alpha_1,i},\dots,\psi_{\alpha_k,i})(t)$ restricted to any neighborhood of $a_i$ equals the range of the function $[-1,1]^k\ni (t_1,\dots,t_k)\mapsto P(t_1,\dots,t_k)$. Thus $t\mapsto P(\psi_{\alpha_1,i},\dots,\psi_{\alpha_k,i})(t)$ is discontinuous at $a_i$. Note that $P(\varphi_{\alpha_1},\dots,\varphi_{\alpha_k})=P(\psi_{\alpha_1,i},\dots,\psi_{\alpha_k,i})$ on $I_i$.

The rest of the proof that $\varphi\in\mathcal{C}_G$ is the same as that of Theorem \ref{c+lineability}.
\end{proof}

The most common example of a function $f$ whose set of continuity points equals $\R\setminus A$ with countable $A=\{a_i:i=1,2,3,\dots\}$ is the following
$$
f(t)=\sum_{a_i<t}\frac{1}{2^i}. 
$$
The function $f$ has a jump discontinuity at each $a_i$. The following proposition shows why we could not use such functions instead of $\psi_{\alpha,i}$ in the proof of Theorem \ref{c+lineability}. 
\begin{proposition}
Let $E\subset\R^\R$ stand for the set of all functions $f$ from $\R^\R$ such that $\lim_{x\to 0^-}f(x)$ and $\lim_{x\to 0^+}f(x)$ exist but $f$ is not continuous at 0. Then $E$ is not 4-lineable.  
\end{proposition}

\begin{proof}
Suppose that $E\cup 0$ contains a $4$-dimensional linear subspace spanned by $f_1,\dots,f_4$. Since the set $\{(\lim_{x\to 0^-}f_i(x),f_i(0),\lim_{x\to 0^+}f_i(x)):i=1,2,3,4\}$ is not linearly independent, there are nonzero $c_1,\dots,c_4$ such that $f=c_1f_1+\dots+c_4f_4$ is continuous at $0$ which yields a contradiction.  
\end{proof}

We are interested in the problem whether $\mathcal C_\Q$ is strongly $\mathfrak{c}$-algebrable. We believe that the answer would help to solve in general the problem of the maximal strong algebrability of $\mathcal C_G$, for arbitrary $G_\delta$ set $G$. Finally let us remark that in \cite{GKP} it was proved that the set of all functions from $\R^\R$ with dense set of jump discontinuities is strongly $\mathfrak{c}$-algebrable.

%

\subsection{Compact preserving Darboux mapping.} 

It is known that a function $f:\R\to\R$ is continuous if and only if for each compact subset $K\subset \R$ the image $f(K)$ is compact and $f$ is Darboux, i.e. for each connected subset $C\subset\R$ the image $f(C)$ is connected. This result was proved by several authors. Recently G\'amez-Merino, Mu\~noz-Fern\'andez and Seoane-Sep\'ulveda \cite{GMMFSS} proved that the family of Darboux nowhere continuous functions  and the family of compact preserving nowhere continuous functions are $2^\mathfrak{c}$-lineable. This results were improved by Bartoszewicz, Bienias and G\l\c{a}b to $2^\mathfrak{c}$-algebrability \cite{BBG}. Note that Theorem \ref{Th2}(iii) says that the family of Darboux nowhere continuous functions is strongly  $2^\mathfrak{c}$-algebrable. The next result deals with the algebrability of compact preserving mapping which are nowhere continuous. 
We will show that also the family of all nowhere continuous functions which map uncountable compact sets to compact sets is also strongly  $2^\mathfrak{c}$-algebrable. However, the family of all nowhere continuous compact preserving functions is $2^\mathfrak{c}$-algebrable but not strongly  1-algebrable. 

\begin{theorem}\label{ThCompact}
\begin{itemize}
\item[(i)] The family of all nowhere continuous Darboux functions which maps uncountable compact sets to compact sets is strongly  $2^\mathfrak{c}$-algebrable;
\item[(ii)] the family of all nowhere continuous compact--preserving functions is $2^\mathfrak{c}$-algebrable but not strongly  1-algebrable.
\end{itemize}
\end{theorem}

\begin{proof}
Let $\mathcal F$ be the family of nonempty perfect subsets of $\R$. The family $\mathcal{SES}([0,1],\mathcal F)$ consists of functions $f$ such that $f$ maps nonempty pefect sets onto compact intervals, $f$ is Darboux nowhere continuous. By Theorem \ref{main1} the family $\mathcal{SES}([0,1],\mathcal F)$ is strongly $2^\mathfrak{c}$-algebrable which proves (i).

The $2^\mathfrak{c}$-algebrability of the family of all nowhere continuous compact--preserving functions was proved in \cite{BBG}. Suppose that $f$ is nowhere continuous compact--preserving functions. Banakh, Bartoszewicz, Bienias and G\l\c{a}b have recently proved that compact--preserving nowhere continuous function cannot take infinitely many values on every interval \cite{BBBG}. Thus there is an interval $I$ such that $f(I)=\{a_{1},...,a_{n}\}$. The set \[\{(a_{1},...,a_{n}),(a^{2}_{1},...,a^{2}_{n}),...,(a^{n+1}_{1},...,a^{n+1}_{n})\}\] is not linearly independent in $\R^{n}$. Hence there are numbers  $\alpha_{1},...,\alpha_{n+1} \in \R$ not vanishing simultaneously and such that $\alpha_{1} a_{k} + ... + \alpha_{n+1} a^{n+1}_{k} =0$ for every $k \in \{1,...,n\}$. Therefore $\alpha_{1} f + ... +\alpha_{n+1} f^{n+1}$ is the zero function on $I$, and therefore continuous on $I$. Hence the family of nowhere continuous compact--preserving functions is not strongly 1-algebrable which proves (ii).
\end{proof}

As we mentioned, compact--preserving Darboux functions are necessarily continuous. As it was remarked in \cite{BBBG} it is enough to assume that Darboux function  $f$ maps countable compact sets onto compact sets to get the continuity of $f$. Theorem \ref{ThCompact}(i) shows that preserving of compactness of uncountable sets and the Darboux property are not sufficient for continuity.


\end{document}